\documentclass[12pt,a4paper,leqno]{article}

\usepackage{amsmath,amstext,amssymb,amscd,euscript}
 \usepackage{graphicx}

\usepackage[english,russian]{babel}
\usepackage[cp1251]{inputenc}

\newcommand{\Xcomment}[1]{}
\newtheorem{theorem}{Теорема}[section]
\newtheorem{lemma}[theorem]{Лемма}
\newtheorem{corollary}[theorem]{Следствие}

\newtheorem{prop}[theorem]{Предложение}

\newenvironment{proof}{\noindent{\bf Доказательство}\/}%
{\hfill$\qed$\medskip}

\def\qed{\Box}

\makeatletter \@addtoreset{equation}{section} \makeatother

\newenvironment{numitem1}{\refstepcounter{equation}\begin{enumerate}%
\item[(\thesection.\arabic{equation})]}{\end{enumerate}}

\newcommand{\refeq}[1]{(\ref{eq:#1})}  

 \makeatletter
\renewcommand{\section}{\@startsection{section}{1}{0pt}%
{-3.5ex plus -1ex minus -.2ex}{2.3ex plus .2ex}%
{\normalfont\Large}}
 \makeatother

 \makeatletter
\renewcommand{\subsection}{\@startsection{subsection}{2}{0pt}%
{-3.0ex plus -1ex minus -.2ex}{1.5ex plus .2ex}%
{\normalfont\normalsize\bf}}
 \makeatother
 
 \newcommand{\SEC}[1]{\ref{sec:#1}}  

\def\Rset{{\mathbb R}}

\def\Ascr{{\cal A}}

\def\Cscr{{\cal C}}

\def\Escr{{\cal E}}
\def\Fscr{{\cal F}}

\def\Iscr{{\cal I}}

\def\Mscr{\EuScript{M}}
\def\Pscr{{\cal P}}

\def\Rscr{{\cal R}}
\def\Sscr{{\cal S}}
\def\Tscr{{\cal T}}

\def\tilde{\widetilde}
\def\hat{\widehat}

\def\eps{\varepsilon}

\def\Mmin{M^{\rm min}}
\def\Mmax{M^{\rm max}}

\def\replac{{\rm Repl}}

\def\Pscrst{\Pscr_{\rm st}}

\newcommand{\crepl}[3]{{#1}\stackrel{#2}{\longrightarrow}{#3}}

\begin{document}

\baselineskip=15pt
\parskip=2pt

\title{On the set of stable matchings of a bipartite graph}

\author{Alexander V.~Karzanov
\thanks{Central Institute of Economics and Mathematics of
the RAS, 47, Nakhimovskii Prospect, 117418 Moscow, Russia; email:
akarzanov7@gmail.com.}
 }

\date{}

 \maketitle

\begin{quote} 
\textbf{Abstract.}
The topic of stable matchings (marriages) in a bipartite graph has become widely popular, starting with the appearance of the classical work by Gale and Shapley. We give a detailed survey on selected known results in this field that demonstrate structural, polyhedral and algorithmic properties of such matchings and their sets, providing our description with relatively short proofs.
(The paper is written in Russian.)
 \end{quote}

\baselineskip=15pt
\parskip=2pt

\section{Введение}  \label{sec:intr}
Начиная с классической работы Гейла и Шепли~\cite{GS}, область задач о стабильных марьяжах и их обобщений привлекала внимание многочисленных исследователей, и в этой области был собран богатый урожай интересных результатов, как теоретического, так и прикладного характера.

В задаче о стабильном марьяже (или, более определенно, о стабильном матчинге в двудольном графе) рассматривается двудольный граф $G=(V,E)$, в котором  для каждой вершины $v$ задан линейный порядок $<_v$, указывающий предпочтения на множестве ребер $\delta_G(v)$, инцидентных $v$. Пусть $I$ и $J$ -- вершинные доли в $G$. В~\cite{GS} и других работах вершины в $I$ и $J$ понимались как лица разных полов (``мужчины'' и ``женщины'', соответственно), а ребра как возможные союзы (или ``браки'')  между ними: если для вершины $m\in I$ (``man'') и инцидентных ей ребер $e=mw$ и $e'=mw'$ выполняется $e<_m e'$, это означает, что $m$ предпочитает союз с $w$ союзу с $w'$. И аналогично для $w\in J$ (``woman'') относительно порядка $<_w$.

\emph{Матчинг} в $G$ -- это подмножество ребер $M\subseteq E$, где никакие два ребра не имеют общей вершины; таким образом, можно считать, что $M$ описывает множество  ``моногамных браков''. Матчинг $M$ называется \emph{стабильным}, если для любого ребра $e=mw\in E-M$ по крайней мере одна из вершин $m$ и $w$ имеет инцидентное ребро в $M$ (определяющее выбранного ``партнера''), и это ребро является более предпочтительным для данной вершины, чем $e$.  Стабильный матчинг существует для любого двудольного $G=(V,E)$ с произвольными линейными порядками $<_v$ на $\delta_G(v)$, $v\in V$, что первоначально было доказано Гейлом и Шепли~\cite{GS} для полных двудольных графов с долями одного размера, и обобщено последующими авторами на произвольные двудольные графы.

Впоследствии развитие в области стабильности, касающейся графов, шло по нескольким направлениям. В одном из них понятие стабильности переносилось на матчинги в произвольных графах (здесь следует особо выделить основополагающие работы Ирвинга~\cite{irv} и Тана~\cite{tan} о ``стабильных расселениях по комнатам'' (stable roommates)). В другом направлении изучались вопросы стабильности при рассмотрении вариантов весовых функций на ребрах и вершинах графа (отметим, как одну из наиболее общих, задачу о ``стабильном распределении'' (stable allocation problem), которая была введена и изучена Бейу и Балински~\cite{BB}). 

В то же время целый ряд глубоких результатов был получен в пределах собственной теории стабильных матчингов для двудольных графов, и цель данной работы -- сделать обзор избранных, наиболее значительных на наш взгляд, известных результатов в этом теории. Они в большой степени связаны между собой и охватывают комбинаторные, полиэдральные, алгоритмические и др. аспекты. Ряд задач, освещаемых в обзоре, состоит в описании структурных свойств множества $\Mscr(G)$ стабильных матчингов в $(G,<)$.

Наше изложение организовано следующим образом.

Раздел~\SEC{basic} содержит описание базовых утверждений и конструкций. Он начинается с напоминания классических результатов, восходящих к работе Гейла и Шепли о существовании стабильного матчинга и о методе ``отложенных принятий'' для его построения. Затем объясняется, что в множестве $\Mscr(G)$ стабильных матчингов произвольного двудольного графа $G$ все матчинги покрывают одно и то же множество вершин (Предложение~\ref{pr:cover_vert}), и описываются операции на парах стабильных матчингов, позволяющих определить частичный порядок $\prec$ на множестве $\Mscr(G)$, представляющий его в виде дистрибутивной решетки (Предложение~\ref{pr:lattice}).

Раздел~\SEC{rotat} посвящен обсуждению ключевого понятия \emph{ротации} в стабильном матчинге. Оно было введено Ирвингом~\cite{irv} в связи с задачей на произвольном графе (stable roommates problem), но оказалось очень полезным и для исследования структуры множества стабильных матчингов двудольного графа $G$. В графе $G$ со стабильным матчингом $M$ мы понимаем под ротацией определенный цикл в $G$, чередующийся относительно $M$. К числу основных из известных утверждений, обсуждаемых в этом разделе, относятся такие: (i) трансформация вдоль ротации преобразует стабильный матчинг в стабильный (Предложение~\ref{pr:rotat-stab}); (ii) ротационные трансформации соответствуют отношениям немедленного предшествования в решетке $(\Mscr(G),\prec)$ (Предложение~\ref{pr:trass}); (iii) во всех максимальных цепях решетки $(\Mscr(G),\prec)$ множество ротаций одно и то же, обозначаемое $\Rscr_G$ (Предложение~\ref{pr:C-tau}). 

В разделе~\SEC{ideal} продолжаются обсуждения, связанные с ротациями. Здесь вводится частичный порядок $\lessdot$ на множестве $\Rscr_G$ и приводится альтернативное представление для множества стабильных матчингов $\Mscr(G)$, установленное Ирвингом и Лейтером~\cite{IL}, а именно: стабильные матчинги в $G$ взаимно-однозначно соответствуют  идеалам (или анти-цепям) посета $(\Rscr_G,\lessdot)$ (Предложение~\ref{pr:ideal-match}).

В разделе~\SEC{optimal} рассматривается усиленный, ``стоимостной'', вариант задачи о стабильном матчинге. Здесь множество ребер двудольного графа $G=(V,E,<)$ снабжаются вещественной весовой функцией $c: E\to\Rset$, и требуется найти стабильный матчинг $M$, минимизирующий общий вес $\sum(c(e)\colon e\in E)$. В частном случае, когда для ребра $e=mw$ вес $c(e)$ равняется сумме его рангов в порядках $<_m$ и $<_w$, получаем задачу об \emph{эгалитарном} стабильном матчинге (в котором стороны $I$ и $J$ ``уравниваются''), поставленную в~\cite{MW71}. Для произвольных весов $c$ задача минимизации (или максимизации) решается эффективным комбинаторным алгоритмом, указанном в~\cite{ILG}. Это является следствием того, что: стабильные матчинги представляются идеалами посета ротаций $(\Rscr_G,\lessdot)$, в котором число элементов $|\Rscr_G)|$ не превосходит числа ребер $|E|$; вес стабильного матчинга выражается, с точностью до константы, как вес соответствующего идеала (при назначении подходящих весов ротаций); и задача о нахождении идеала (или ``замкнутого множества'') минимального веса в произвольном конечном посете сводится к задаче о минимальном разрезе ориентированного графа, как показано Пикаром~\cite{pic}.

Раздел~\SEC{polyhedr} посвящен полиэдральным аспектам. Здесь описывается многогранник стабильных матчингов $\Pscrst(G)$ через  систему линейных неравенств (близкую к той, что указана Ванде Вейтом~\cite{VV}). Также, следуя полиэдральной конструкции Тео и Сетарамана~\cite{TS}, показывается, что для произвольного набора стабильных матчингов $M_1,\ldots,M_\ell$ и любого $k\le \ell$, выбирая для каждой покрытой вершины в доле $I$ ребро, имеющее порядок $k$ среди этих матчингов, мы снова получаем стабильный матчинг. В частности, когда $\ell$ нечетно и $k=(\ell+1)/2$, определяется т.н. \emph{медианный} стабильный матчинг для заданных $M_1,\ldots,M_\ell$.

В заключительном разделе~\SEC{addit} формулируется результат, полученный в~\cite{ILG}, о труднорешаемости ($\#P$-полноте) задачи определения числа стабильных матчингов в двудольном графе.  Это отвечает на вопрос Кнута~\cite{knu} об алгоритмической сложности вычисления такого числа. 

Для полноты изложения мы, как правило, сопровождаем представленные утверждения доказательствами, часто альтернативными и более короткими по сравнению с оригинальными доказательствами в работах авторов. Отметим, что везде, где мы даем ссылки на работы~\cite{IL}, \cite{ILG} и некоторые другие, соответствующие результаты в этих работах были получены для случая полного двудольного графа $K_{n,n}$ (с произвольными порядками $<$), но эти результаты без большого труда распространяются на случай произвольного двудольного графа $G$.


\section{Определения и базовые факты}  \label{sec:basic}

На протяжении всего изложения мы рассматриваем двудольный граф $G=(V,E)$ с разбиением множества вершин $V$ на подмножества $I$ и $J$ (где каждое ребро соединяет вершины из разных подмножеств). Иногда мы будем считать граф $G$ ориентированным, с ребрами направленными от $I$ к $J$. Ребро, соединяющее вершины $m\in I$ (``man'') и $w\in J$ (``woman''), обозначается $mw$. Множество ребер, инцидентных вершине $v\in V$ обозначается $\delta(v)=\delta_G(v)$. Каждая вершина $v$ снабжена линейным порядком $<_v$, задающим \emph{предпочтения} на множестве $\delta(v)$. А именно для ребер $e,e'\in \delta(v)$, если $e<_v e'$, то ребро $e$ считается более предпочтительным для $v$, чем ребро $e'$; иногда нам также будет удобно говорить, что $e$ расположено \emph{раньше} или \emph{левее}, чем $e'$ (а $e'$ позднее, или правее, чем $e$). Если вершина $v$ ясна из контекста, можно писать $e<e'$. Обычно мы включаем порядки $<$, а также разбиение $V$ на доли $I,J$ в описание графа, применяя обозначение $G=(V,E,<)$ или $G=(I\sqcup J,E,<)$.

Матчинг $M$ в $G$ называется \emph{стабильным}, если он не допускает блокирующих ребер. Здесь ребро $e=mw\in E-M$ считается \emph{блокирующим} для $M$, если вершина $m$  либо не покрывается $M$, либо ребро $e'$ в $M$, инцидентное $m$, менее предпочтительно: $e<_m e'$, и то же самое верно для вершины $w$. 

Ниже мы приводим две группы базовых свойств стабильных матчингов в $G$.
\medskip

\noindent I. В основании теории стабильных матчингов лежит эффективный алгоритм построения стабильного матчинга Гейла и Шепли~\cite{GS} (в изложении для полного двудольного графа с долями $I,J$ одинакового размера $n$: $G\simeq K_{n,n}$). Рекурсивное изложение этого алгоритма, который часто называют ``алгоритмом отложенных принятий'' (АОП), дано МакВайти и Вилсоном в работе~\cite{MW71} (где также описан алгоритм последовательного конструирования множества всех стабильных матчингов для $G\simeq K_{n,n}$). Короткое неконструктивное доказательство существования стабильного матчинга в двудольном случае можно найти в~\cite[Sec.~18.5g]{sch}.

В АОР роли сторон $I$ и $J$ неравнозначны: одна сторона ``предлагает, а другая ``принимает или отвергает'' предложения. (Напомним, что на каждом шаге АОП произвольно выбирается не включенный в текущий марьяж ``man'' $m\in I$, который делает предложение ``woman'' $w\in J$ согласно наилучшему ребру $mw$ из еще не использованных им ребер в $\delta(m)$; это предложение сразу отвергается, если $w$ уже имеет лучшее предложение от $m'\in I$ (т.е. $m'w<_w mw$), и принимается иначе, отвергая предыдущее предложение, если таковое имеется.) Результаты естественно расширяются и для произвольного двудольного графа $G$.  Это приводит к следующему важному свойству. 

\begin{prop} \label{pr:GSalg} {\rm \cite{GS},\cite{MW71}}
В случае, когда $I$ предлагает, а $J$ принимает/отвергает, АОП находит (за линейное время $O(|V|+|E|)$) канонический стабильный матчинг, в котором выбор ребер для вершин в $I$ наилучший, а для вершин в $J$ наихудший по всем стабильным матчингам. 
 \end{prop}

Мы обозначаем этот матчинг $\Mmin=\Mmin(G)$. По симметрии, если АОП применяется в варианте, когда $J$ предлагает, то он строит стабильный матчинг, для которого выборы ребер для $J$ наилучшие, а для $I$ наихудшие; обозначим его $\Mmax=\Mmax(G)$. Уточним смысл терминов ``наилучший-наихудший'', рассматривая любой другой стабильный матчинг $L$ для $G$. А именно, если ребра $e\in \Mmin$ и $e'\in L$ имеют общую вершину $m\in I$, то $e\le_m e'$, а если общую вершину $w\in J$, то $e'\le_w e$. Здесь мы опираемся на инвариантность  множества вершин, покрытых стабильным матчингом (что тривиально в случае полного двудольного графа $K_{n,n}$, где стабильный матчинг является совершенным, т.е. покрывающим все вершины). А именно, справедливо следующее свойство, установленное в~\cite{MW70}.

\begin{prop} \label{pr:cover_vert}
Для двудольного $G=(V,E,<)$ все стабильные матчинги покрывают одно и то же множество вершин.
  \end{prop}

Чтобы показать это свойство, примем следующие определения и обозначения (которые понадобятся и в дальнейшем).

\emph{Путь} в ориентированном графе -- это последовательность $P=(v_0,e_1,v_1,\ldots,e_k,v_k)$,  где $e_i$ -- ребро, соединяющее вершины $v_{i-1}$ и $v_i$. Для пути $P$ мы можем также применять сокращенную запись через вершины: $v_0v_1\cdots v_k$, или ребра: $e_1e_2\cdots e_k$.  Ребро $e_i$ в $P$ называется  \emph{прямым} (forward), если $e_i=v_{i-1}v_i$, и \emph{обратным} (backward), если $e_i=v_iv_{i-1}$. (Мы обозначаем ребро, выходящее из $u$ и входящее в $v$, как $uv$, вместо $(u,v)$.) Путь называется \emph{ориентированным}, если все его ребра прямые, и называется \emph{простым}, если все его вершины различны. Путь из вершины $u$ в вершину $v$ может быть назван $u$--$v$ \emph{путем}.
Для подмножества ребер $A\subseteq E$ множество вершин, покрытых $A$, обозначается $V_A$. Для подмножеств $X,Y$ множества $S$ мы пишем $X-Y$ вместо $X\setminus Y$ ($=\{s\in X\colon s\notin Y\}$) и обозначаем $X\triangle Y$ симметрическую разность $(X-Y)\cup (Y-X)$.
 \smallskip

Пусть теперь $M,L$ -- два стабильных матчинга в двудольном $G$. Рассмотрим подграф $\Delta_{M,L}$ в $G$, индуцированный множеством ребер $M\triangle L$. Заметим, что

\begin{numitem1} \label{eq:3path}
если $e,e',e''$ -- различные ребра в $\Delta_{M,L}$ такие, что $e,e''$ имеют общую вершину $u$, а $e',e''$ имеют общую вершину $v$, и если $e>_u e'$, то $e'>_v e''$.
  \end{numitem1}

Действительно, пусть для определенности $e'\in M$. Тогда $e,e''\in L$. Если бы было $e'<_v e''$, то ребро $e'$ было бы блокирующим для $L$. 
\smallskip

Перейдем теперь к доказательству Предложения~\ref{pr:cover_vert}. Предположим, что $V_M\ne V_L$. Тогда по крайней мере одна из компонент в $\Delta_{M,L}$ является путем $P=(v_0,e_1,v_1,\ldots,e_k,v_k)$ (при $v_0\ne v_k$). В этом пути ребра из $M$ и $L$ чередуются, и из стабильности $M,L$ следует $k\ge 2$. Без ограничения общности можно считать, что $e_1\in M$. 
Тогда $v_0\notin V_L$, и так как ребро $e_1$ не является блокирующим для $L$, то должно выполняться $e_1>_{v_1} e_2$. 

Ввиду этого, последовательно применяя~\refeq{3path} к тройкам соседних ребер в $P$, получаем, что $e_{k-1}>_{v_{k-1}} e_k$. Но тогда при нечетном $k$ ребро $e_k$ является блокирующим для $L$ (так как $e_{k-1}\in L$, $e_k\notin L$ и $v_k\notin V_L$), а при четном $k$ ребро $e_k$ является блокирующим для $M$ (так как $e_{k-1}\in M$, $e_k\notin M$ и $v_k\notin V_M$); противоречие. \hfill$\qed$
 \medskip

Будем обозначать множество вершин в $G$, покрытых стабильным матчингом, через $\tilde V$. Очевидно, при удалении вершин в $V-\tilde V$ каждый стабильный матчинг для $G$ остается стабильным и для результирующего графа $\tilde G$. Однако в $\tilde G$ могут появиться новые стабильные матчинги (и, следовательно, включение $\Mscr(G)\subseteq \Mscr(\tilde G)$ может быть строгим), как показывает простой пример в левом фрагменте на Рис.~\ref{fig:uncover}  Здесь отношения предпочтений устроены так: $a<b$, $b<c$, $c<d$, $d<e<a$, и можно проверить, что имеется единственный стабильный матчинг, а именно, $M=\{b,d\}$. Тогда вершина  $v$ не покрыта $M$, однако при ее удалении (вместе с ребром $e$) появляется новый стабильный матчинг $\{a,c\}$. 

 \begin{figure}[htb]
\begin{center}
\includegraphics[scale=0.8]{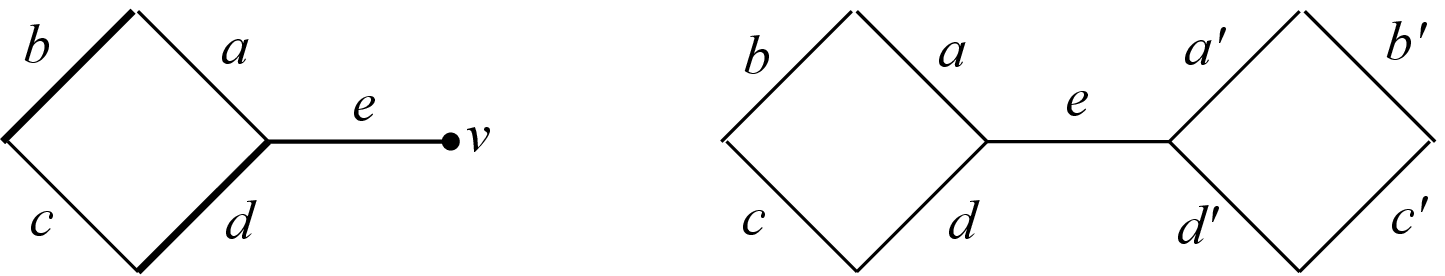}
\end{center}
\vspace{-0.4cm}
\caption{Два примера}
 \label{fig:uncover}
\end{figure}

Таким образом, при исследовании множества стабильных матчингов для $G$ мы, вообще говоря, не можем выкидывать из рассмотрения непокрытые вершины, оставляя только подграф $\tilde G=(\tilde V=(\tilde I\sqcup \tilde J), \tilde E)$ (где доли $\tilde I=I\cap \tilde V$ и $\tilde J=J\cap\tilde V$ имеют одинаковый размер, и все стабильные матчинги совершенные). По похожим причинам в случае отсутствия непокрытых вершин удаление ребра, не входящего ни в один стабильный матчинг, может повлечь появление нового стабильного матчинга. (В правом фрагменте на Рис.~\ref{fig:uncover} изображено расширение предыдущего графа с аналогичными предпочтениями для новых ребер, а именно, $a'<b'$, $b'<c'$, $c'<d'$ и $d'<e<a'$. Здесь имеются три стабильных матчинга, а именно, $\{a,c,b'd'\}$, $\{b,d,a',c'\}$ и $\{b,d,b',d'\}$, однако при удалении непокрытого ребра $e$ появляется четвертый стабильный матчинг $\{a,c,a',c'\}$, для которого ранее имелось блокирующее ребро $e$.)
\medskip

\noindent II. Далее рассмотрим упорядоченную пару $(M,L)$ стабильных матчингов в $G$.  Из Предложения~\ref{pr:cover_vert} следует, что граф $\Delta_{M,L}$, индуцированный $M\triangle L$,  распадается на некоторое множество $\Cscr=\Cscr(M,L)$ непересекающихся чередующихся циклов. Считая $G$ ориентированным от $I$ к $J$, будем полагать каждый цикл  $C=(v_0,e_1,v_1,\ldots,e_k,v_k=v_0)\in\Cscr$  направленным согласно ориентации ребер в $L$, т.е. прямые ребра в $C$ принадлежат $L$, а обратные принадлежат $M$. Ввиду~\refeq{3path}, все предпочтения вдоль $C$  ``направлены в одну сторону'', а именно,

\begin{numitem1} \label{eq:cycleC}
для цикла $C=e_1e_2\cdots e_k\in\Cscr(M,L)$ (применяя обозначение $C$ через ребра), если $e_i<e_{i+1}$ выполнено для некоторого $i$, то это выполняется для всех $i=1,\ldots,k$ (полагая $e_{k+1}:=e_1$).
 \end{numitem1}
 
В этом случае скажем, что цикл $C$ \emph{повышающий}, или \emph{правый}, относительно $M$. Иначе $C$ считается \emph{понижающим}, или \emph{левым}. Обозначим $\Cscr^+(M,L)$ и $\Cscr^-(M,L)$ множества правых и левых циклов для $(M,L)$, соответственно.
Если матчинг $M'$ получается из $M$ заменой ребер вдоль цикла $C$ будем писать $\crepl{M}{C}{M'}$ или $M'=\replac(M,C)$, и аналогично, будем писать $\crepl{L}{C}{L'}$ или $L'=\replac(L,C)$ для замены вдоль $C$ в $L$. Заметим, что если цикл $C$ правый относительно $M$, то по сравнению с $M$ матчинг $M'$ является менее предпочтительным для всех вершин в $I\cap C$ (``men'') и более предпочтительным для всех вершин в $J\cap C$ (``women''), а для $L$ и $L'$ поведение противоположное. (Иначе говоря, при замене вдоль повышающего цикла мы как бы удаляемся от $\Mmin$ и приближаемся к $\Mmax$.) Повышающий цикл показан на Рис.~\ref{fig:cycle}.

 \begin{figure}[htb]
\begin{center}
\includegraphics[scale=0.9]{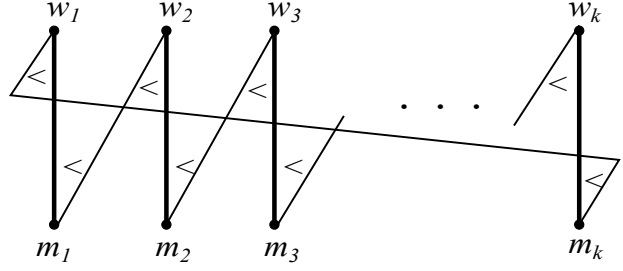}
\end{center}
\vspace{-0.5cm}
\caption{Повышающий цикл для $(M,L)$. Матчинг $M$ изображен жирным, а $L$ -- тонким. Вершины $m_i$ принадлежат доле $I$.}
 \label{fig:cycle}
\end{figure}

Можно было бы ожидать, что матчинг $M'=\replac(M,C)$ должен быть стабильным. Однако это может быть неверным. Например, рассмотрим граф как в правом фрагменте на Рис.~\ref{fig:uncover} (с указанными выше предпочтениями) и возьмем стабильные матчинги $M=\{b,d,a',c'\}$ и $L=\{a,c,b',d'\}$. Тогда ребра $a,b,c,d$ порождают цикл $C$ в $\Delta_{M,L}$, но матчинг $M'=\{a,c,a',c'\}$, получающийся в результате замены ребер в $M$ вдоль $C$, не является стабильным.

Тем не менее, стабильность сохраняется при замене сразу во всех циклах в $\Cscr^+(M,L)$ или в $\Cscr^-(M,L)$. В этих случаях будем писать $M'=\replac(M,\Cscr^+(M,L))$ и $M'=\replac(M,\Cscr^-(M,L))$, соответственно.

\begin{lemma} \label{lm:replML} {\rm \cite{knu}}
Если $M,L$ -- стабильные матчинги, то $M'=\replac(M,\Cscr^+(M,L))$ -- тоже стабильный матчинг. Аналогично для $\replac(M,\Cscr^-(M,L))$.
 \end{lemma}
 
\begin{proof}
Очевидно, $M'$ является матчингом. Предположим, $M'$ имеет блокирующее ребро $e=mw$ (где $m\in I$). Легко видеть, что это в принципе возможно только если одна из вершин $m,w$ принадлежит правому циклу, а другая -- левому циклу в $\Cscr(M,L)$. Без ограничения общности можно считать, что $m$ принадлежит циклу $C\in\Cscr^+(M,L)$, а $w$ -- циклу  $C'\in\Cscr^-(M,L)$. Пусть вершина $m$ инцидентна  ребрам $a\in M$ и $b\in L$ в $C$, а $w$ инцидентна ребрам $c\in M$ и $d\in L$ в $C'$. Тогда $a<_m b$ и $c<_w d$ (учитывая $m\in I$ и $w\in J$). При преобразовании $M\mapsto M'$ в матчинг $M'$ попадают ребра $b$ и $c$, и должно выполняться $e<_m b$ и $e<_w c$ (так как по предположению $e$ блокирует $M'$). Но тогда получаем $e< d$ (ввиду $c<d$), и, следовательно, $e$ блокирует $L$; противоречие.
\end{proof}

Это позволяет определить решетку на стабильных матчингах в $G$. Заметим, что для $M,L\in\Mscr(G)$ заменить в $M$ ребра вдоль всех повышающих циклов для $M$ -- это все равно, что сделать в $L$ замены вдоль всех повышающих циклов для $L$, т.е. $\replac(M,\Cscr^+(M,L))=\replac(L,\Cscr^+(L,M))$. Аналогично $\replac(M,\Cscr^-(M,L))=\replac(L,\Cscr^-(L,M))$. Ввиду Предложения~\ref{pr:GSalg}, в $\Cscr(M,L)$ нет понижающих циклов при $M=\Mmin$, и нет повышающих циклов при $M=\Mmax$.

 \begin{prop} \label{pr:lattice} {\rm \cite{knu}}
Для различных $M,L\in \Mscr(G)$ положим $M\prec L$, если для каждой пары ребер $e\in M$ и $e'\in L$, инцидентных вершине в $\tilde I$ ($=I\cap \tilde V$) выполняется $e\le e'$ (давая противоположные соотношения для вершин в $\tilde J$). Тогда $\prec$ определяет дистрибутивную решетку на $\Mscr(G)$ с минимальным элементом $\Mmin (G)$ и максимальным элементом $\Mmax (G)$, в которой для $M,L\in\Mscr(G)$ точной нижней гранью $M\wedge L$ является $\replac(M,\Cscr^-(M,L))=\replac(L,\Cscr^-(L,M))$, а точной верхней гранью $M\vee L$ --  $\replac(M,\Cscr^+(M,L))=\replac(L,\Cscr^+(L,M))$. \hfill$\qed$
 \end{prop}
 

\section{Ротации}  \label{sec:rotat}

Благодаря решеточности множества $\Mscr(G)$, можно, имея два или более стабильных матчингов, конструировать новые, делая переназначения в соответствующих циклах для пар матчингов, как указано в Лемме~\ref{lm:replML}. Другой метод, связанный с понятием ротации, позволяет порождать новые стабильные матчинги, исходя из одиночных элементов в $\Mscr(G)$. Это понятие было введено Ирвингом и Лейтером~\cite{IL} для задачи о стабильном марьяже, основываясь на концепции ``цикла типа все или ничего'' (all-or-nothing cycle) из работы Ирвинга~\cite{irv}, посвященной задаче о стабильных матчингах в общем графе (т.н. ``stable roommates problem''). Впоследствии понятие ротации было распространено и на другие задачи о стабильности, такие как задачи о стабильных b-матчингах, распределениях (allocations) и др. (см., например,~\cite{DM}). Заметим, что хотя в~\cite{IL} рассматривался случай полного двудольного графа $G\simeq K_{n,n}$, конструкции и результаты без особого труда переносятся на произвольный случай, и как выше, мы далее будем рассматривать произвольный двудольный граф $G=(V=I\sqcup J,\, E,<)$.

Пусть $M$  -- стабильный матчинг в $G$.
\medskip

\noindent\textbf{Определения.} 
Скажем, что ребро $a=mw\in E-M$ (где $m\in I$) \emph{допустимое}, если в $M$ есть ребро $e$, инцидентное $m$, и это ребро более предпочтительное, чем $a$ (т.е. $e<_m a$), и при этом либо вершина $w$ непокрытая, либо $M$ имеет ребро $e'$, инцидентное $w$, и это ребро менее предпочтительное, чем $a$ (т.е. $a<_w e'$).  Если для вершины $m\in I$ множество допустимых ребер в $\delta(m)$ непустое, то самое первое (наиболее предпочтительное) из них назовем \emph{активным}. Обозначим $A$ множество активных ребер для $M$. Подграф в $G$, индуцированный  множеством ребер $M\cup A$, обозначим $\Gamma=\Gamma(M)$ и назовем \emph{активным графом}. 
\smallskip

Учитывая то, что из каждой вершины в $I$ исходит не более одного активного ребра (в то время как в вершину в $J$ может входить несколько активных ребер), всякий цикл в $\Gamma$ должен быть чередующимся относительно $M$ и $A$. Кроме того, каждая компонента $K$ в $\Gamma$ либо является деревом, либо содержит ровно один цикл. (Действительно, в противном случае в $K$ имелась бы пара циклов и соединяющий их простой $u$--$v$ путь $P$ (все вершины которого, кроме $u,v$, не принадлежат циклам). Концевые ребра этого пути должны принадлежать $A$, и по крайней мере одна из вершин $u,v$ должна находиться в $I$. Эта вершина имеет два инцидентных ребра в $A$ (одно на $P$, другое на цикле); противоречие.) При наличие цикла мы называем компоненту $K$ \emph{цикло-содержащей}, а сам цикл обозначаем через $C_K$ и называем \emph{ротацией} для $M$. (Заметим, что в~\cite{IL} ротацией называется не сам цикл, а его пересечение с $M$.)  Ребра в $C_K$, принадлежащие $M$, назовем \emph{матчинговыми}, а остальные -- \emph{активными}. При замене ребер вдоль ротации $C$ возникает матчинг, который по сравнению с $M$ менее предпочтителен для вершин в $I\cap C$ и более предпочтительный для вершин в $J\cap C$. (Для операции замены вдоль $C$ в~\cite{IL} применяется термин ``rotation elimination''.) Когда $C$ -- цикл в $\Gamma(M)$, мы также можем говорить, что матчинг $M$ \emph{допускает} ротацию $C$.
 \medskip
 
\noindent\textbf{Пример.} На Рис.~\ref{fig:three} изображены три стабильных матчинга $M_1,M_2,M_3$ (выделенных жирно) в графе $G=(V,E,<)$ как на правом фрагменте Рис.~\ref{fig:uncover}. Здесь доля $I$ составлена вершинами $1,2,3,4$, а доля $J$ -- вершинами $x,y,v,w$, и предпочтения заданы как и раньше, а именно:
   \begin{gather*}
  a<_y b,\;\; b<_1 c,\;\; c< _x d,\;\; d<_2 e<_2 a, \\
  a'<_3 b',\;\; b'<_w c',\;\; c'< _4 d',\;\; d'<_v e<_v a'.
   \end{gather*}
Для матчинга $M_1$ все нематчинговые ребра $a,c,e,b',d'$ являются допустимыми, но из них активными оказываются только $c,e,b',d'$ (так как $e<_2 a$). Поэтому подграф $\Gamma(M_1)$ порожден $M_1\cup\{c,e,b',d'\}$, имеет ровно одну компоненту, и она является цикло-содержащей с циклом $C=a'b'c'd'$. При замене вдоль ротации $C$ получаем матчинг $M_2$ (являющийся стабильным согласно Предложению~\ref{pr:rotat-stab} ниже). Для $M_2$ допустимыми ребрами в $E-M_2$ являются только $a$ и $c$ (а $e,a',c'$ -- не допустимые, ввиду $d'<_v e,a'$ и $b'<_w c'$, учитывая $b',d'\in M_2$ и $v,w\in J$). Поэтому активными ребрами являются $a$ и $c$, и граф $\Gamma(M_2)$ состоит из трех компонент, порожденных $\{a,b,c,d\}$, $\{b'\}$ и $\{d'\}$. Первая из них образует цикл $D=abcd$. Наконец, при замене вдоль ротации $D$ получаем стабильный матчинг $M_3$. Для него ни одно из ребер в $E-M_3$ не является допустимым, и $\Gamma(M_3)$ составлено просто ребрами в $M_3$. Можно видеть, что $M_1=\Mmin$ и $M_3=\Mmax$.
 \medskip

 \begin{figure}[htb]
\begin{center}
\includegraphics[scale=0.8]{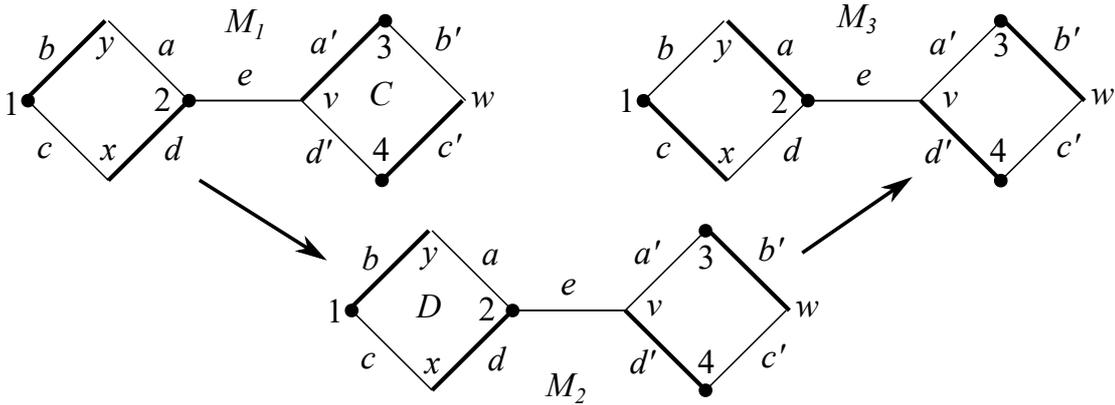}
\end{center}
\vspace{-0.5cm}
\caption{Граф с тремя стабильными матчингами, выделенными жирно: $M_1=\Mmin$ (слева), $M_2$ (снизу) и $M_3=\Mmax$ (справа).}
 \label{fig:three}
\end{figure}

Следующие два утверждения демонстрируют важные свойства $\Gamma(M)$.

\begin{prop} \label{pr:rotat-stab} {\rm\cite{IL}} 
Для любого цикла (ротации) $C$ в $\Gamma(M)$ матчинг $M'=\replac(M,C)$ является стабильным.
 \end{prop}

\begin{proof}
~Оба $M$ и $M'$ покрывают одно и то же множество вершин $\tilde V$. Предположим, что $M'$ допускает блокирующее ребро $e=mw\in E-M'$.  Заметим, что $e\notin M$ (учитывая то, что ребра в $C\cap M$ не могут быть блокирующими для $M'$). Возможны два случая.
 \smallskip
 
\noindent\emph{Случай 1}: ~$w\in \tilde V$. Тогда вершина $w$ принадлежит ребрам $a\in M'$ и $b\in M$, и мы имеем $e<_w a$ (ввиду блокирования) и $a\le_w b$ (так как в случае $a\ne b$ ребро $a$ лежит в $C$ и является допустимым). Следовательно, $e<_w b$. Тогда $m\in \tilde V$ (иначе $e$ было бы блокирующим для $M$).

Пусть вершина $m$ принадлежит ребрам $c\in M'$ и $d\in M$. Тогда $e<_m c$, и должно выполняться $d<_m e$ (иначе $e$ было бы блокирующим для $M$, учитывая $e<_w b$). Следовательно, $e$ является допустимым для $M$ и более предпочтительным, чем ребро $c$, которое принадлежит $A$; противоречие.
\smallskip

\noindent\emph{Случай 2}: ~$w\notin \tilde V$. Тогда $m\in \tilde V$. Пусть $m$ принадлежит ребрам $c\in M'$ и $d\in M$. Так как $e$ является блокирующим для $M'$, но не для $M$, имеем $d<_m e<_m c$. Но тогда ребро $e$ допустимое и более предпочтительное, чем активное ребро $c$; противоречие.
\end{proof}

Говоря далее о чередующихся (относительно $M$ и $A$) путях и циклах в $\Gamma(M)$, мы считаем, что в них выбрано такое направление, чтобы ребра из $A$ были прямыми, а ребра из $M$ обратными.

\begin{prop} \label{pr:tree-comp}
Пусть $e$ -- ребро в $M$, принадлежащее древесной компоненте $K$ в $\Gamma(M)$. Тогда для любого стабильного матчинга $L\in\Mscr(G)$ либо $e\in L$, либо $e$ содержится в понижающем цикле в $\Cscr(M,L)$.
 \end{prop}
 
\begin{proof}
~Предположим, ребро $e=mw\in M$ из древесной компоненты $K$ находится в повышающем цикле $C\in\Cscr^+(M,L)$ для некоторого $L\in \Mscr(G)$. Пусть $C$ имеет вид $e_1e_2\cdots e_k$ (применяя обозначение через ребра и учитывая направленность $C$), и пусть $e:=e_1\in M$. Возьмем максимальный чередующийся путь $P=p_1p_2\cdots  p_r$ (через ребра) в $\Gamma(M)$, начиная с $p_1=e$. Можно считать, что все ребра в $P\cap M$, кроме $p_1$, не лежат в повышающих циклах в $\Cscr(M,L)$. Ребра $e_1$ и $e_2=mw'$ инцидентны вершине $m\in I$, и поскольку $C$ повышающий, $e_1<_m e_2$ и $e_2<_{w'} e_3$. Кроме того, $e_1,e_3\in M$ и $e_2\in L$. Следовательно, $e_2$ -- допустимое ребро для $M$, и в $\delta(m)$ есть активное ребро $a=mv$, для которого $e_1<_m a\le_m e_2$. Это $a$ совпадает с ребром $p_2$ в $P$. Кроме того, $a\ne e_2$ (иначе $e_3\in P$ и $e_3\in C \cap M$, вопреки условию на $p_1=e$). 

Мы утверждаем, что ребро $a=p_2$ -- блокирующее для $L$. Это следует из $a<_m e_2\in L$, если $r=2$ (в этом случае концевая вершина $v$ в $P$ является непокрытой), а также если $r\ge 3$ и $p_3\in L$ (так как  $a<_v p_3$, ввиду $p_3\in M$). Наконец, пусть $p_3\notin L$. Тогда $p_3$ принадлежит  понижающему циклу $C'\in\Cscr^-(M,L)$. Для ребра $b$ в $C'$, инцидентного $v$ и отличного от $p_3$, выполняется $b\in L$ и $p_3<_v b$ (так как $C'$ понижающий). Тогда из $a<_vp_3<_vb$ и $a<_m e_2\in L$ следует, что $a$ блокирует $L$; противоречие.
\end{proof}
 
Ясно, что для матчинга $M'=\replac(M,C)$, полученного заменой вдоль ротации $C$ в $\Gamma(M)$, выполняется $M\prec M'$, где $\prec$ -- частичный порядок в решетке на $\Mscr(G)$ (см. Предложение~\ref{pr:lattice}). Как было установлено в~\cite{IL}, начиная с $\Mmin$ и используя ротации, можно исчерпать все $\Mscr(G)$. 
\medskip

\noindent\textbf{Определение.} ~Назовем \emph{трассой} последовательность стабильных матчингов $\tau=(M_0,M_1,\ldots,M_N)$ в $G$ такую, что $M_0=\Mmin$, $M_N=\Mmax$, и каждое $M_i$ ($1\le i\le N$) получается из $M_{i-1}$ заменой вдоль одной ротации $C_i$ в $\Gamma(M_{i-1})$. Последовательность ротаций $C_1,\ldots,C_N$ обозначим $\Rscr(\tau)$.
 
\begin{prop} \label{pr:trass}
Трассы покрывают всю решетку $(\Mscr(G),\prec)$.
 \end{prop}
 
\begin{proof}
Достаточно показать, что если матчинги $M,L\in\Mscr$ удовлетворяют $M\prec L$, то найдется ротация $C$ для $M$ такая, что для $M'=\replac(M,C)$ справедливо $M'\preceq L$.  Чтобы построить такую ротацию, возьмем вершину $m\in \tilde I$, для которой инцидентные ребра $e=mw\in M$ и $f=mv\in L$ различны. Тогда $e<_m f$ (ввиду $M\prec L$), и ребро $f$ допустимое для $M$ (так как имеется ребро $b\in M$, инцидентное $v$, и для него выполняется $f<_{v} b$, ввиду $M\prec L$ и $w\in \tilde J$). Поэтому в $\delta(m)$ есть активное ребро $a=mw'$, и выполняется $a\le_m f$. Заметим, что $w'\in \tilde V$ (иначе должно быть $a\notin L$; тогда $a<_m f$, и $a$ блокирует $L$).

Следовательно,  имеется ребро $e'=m'w'\in M$, а также ребро $f'\in L$, инцидентное $m'$. Мы утверждаем, что $f'\ne e'$. Действительно, в случае $f'=e'$, мы имели бы $a<_{w'} f'$ и $a<_m f$, и, следовательно, $a$ блокировало бы $L$. 

Теперь из $f'\ne e'$ получаем $e'<_{m'} f'$, аналогично $e<_m f$ в начальной ситуации. Продолжая построение, мы получаем ``неограниченный'' чередующийся путь $P$ в $\Gamma(M)$ с последовательностью ребер $e=e_1=m_1w_1$, $a=a_1=m_2w_1$, $e'=e_2=m_2w_2$, $a_2=m_2w_3,\ldots$\,, и для каждого $i$ имеется ребро $f_i\in L$, инцидентное  $m_i$ и отличное от $e_i$. Тогда $e_i<_{m_i}a_i\le_{m_i} f_i$, и $P$ содержит искомую ротацию $C$.
\end{proof}

Назовем \emph{рангом} матчинга $M$ сумму $\rho(M)$ порядковых номеров его ребер $mw$ в множествах $\delta(m)$ (упорядоченных согласно $<_ m$), где $m\in I$. Очевидно, $\rho(M)\le |E|$, и если $M\prec L$, то $\rho(M)<\rho(L)$. Поэтому
  \begin{numitem1} \label{eq:trass-length}
любая трасса $\tau$ в $G=(V,E,<)$ содержит не более $|E|$ матчингов; следовательно, $|\Rscr(\tau)|<|E|$.
 \end{numitem1}
 
Из Предложений~\ref{pr:rotat-stab}--\ref{pr:trass} мы также получаем

\begin{corollary} \label{cor:min-max}
{\rm (i)} Граф $\Gamma(\Mmax(G))$ является лесом (и ротации для $\Mmax(G)$ отсутствуют).

{\rm(ii)} Ребра минимального стабильного матчинга $M=\Mmin(G)$, находящиеся в древесных компонентах активного графа $\Gamma(M)$, принадлежат всем стабильным матчингам в $\Mscr(G)$.

{\rm(iii)} $G$ имеет единственный стабильный матчинг тогда и только тогда, когда $\Gamma(\Mmin)$ является лесом.
 \end{corollary}

Следующее утверждение, установленное в~\cite{IL}, несмотря на относительную простоту доказательства, является наиболее существенным в области ротаций.

\begin{prop} \label{pr:C-tau}
Множество ротаций $\Rscr(\tau)$ для всех трасс $\tau$ в $(\Mscr(G),\prec)$ одно и то же.
 \end{prop}

\begin{proof}
~Для $M\in\Mscr(G)$ определим $\Tscr(M)$ как множество отрезков трасс (``подтрасс''), начинающихся с $M$ и оканчивающихся в $\Mmax$. Для такой подтрассы $\tau\in\Tscr(M)$ соответствующее множество ротаций обозначим через $\Rscr(\tau)$ и скажем, что матчинг $M$ \emph{особый}, если в $\Tscr(M)$ имеются подтрассы $\tau$ и $\tau'$ с $\Rscr(\tau)\ne \Rscr(\tau')$. Надо доказать, что $\Mmin$ не является особым.

Предположим, что это не так, и рассмотрим особый матчинг $M$ максимальной высоты, т.е. такой, что матчинг $L$ не является особым, если $M\prec L$. Пусть $\Cscr(M)$ -- множество ротаций в $\Gamma(M)$; тогда для любой подтрассы $\tau=(M=M_1,M_2,\ldots,M_k=\Mmax)$ матчинг $M_1$ получается из $M$ заменой вдоль некоторой ротации $C\in \Cscr(M)$. Ввиду максимальности $M$, найдутся ротации $C',C''\in \Cscr(M)$ такие, что матчинги $M'=\replac(M,C')$ и $M''=\replac(M,C'')$ неособые, и множества ротаций $\Rscr(\tau')$ и $\Rscr(\tau'')$ различны для подтрасс $\tau',\tau''\in \Tscr(M)$ таких, что $\tau'$ проходит через $M'$,  а $\tau''$ проходит через $M''$. 

Теперь заметим, что поскольку ротации $C'$ и $C''$ не пересекаются, то $C'$ является ротацией в $\Gamma(M'')$, а $C''$ -- ротацией в $\Gamma(M')$. Поэтому в $\Tscr(M)$ есть подтрассы $\tau',\tau''$ такие, что $\tau'$ начинается с $M,M',\replac(M',C'')$, и $\tau''$ начинается с $M,M'',\replac(M'',C')$, а после матчинга $\replac(M',C'')=\replac(M'',C')$ подтрассы $\tau'$ и $\tau''$ совпадают. Но тогда $\Rscr(\tau')=\Rscr(\tau'')$; противоречие.
\end{proof} 

Это множество ротаций $\Rscr(\tau)$, не зависящее от трасс $\tau\in\Tscr(\Mmin)$, обозначим $\Rscr_G$. Ввиду~\refeq{trass-length}, получаем

\begin{corollary} \label{cor:RscrG}
Число ротаций в $\Rscr_G$ не превосходит $|E|$.
\end{corollary} 

Приведем еще одно полезное свойство:

\begin{numitem1} \label{eq:different}
множества матчинговых ребер двух различных ротаций не пересекаются (и поэтому общее число ребер в ротациях не превосходит $2|E|$).
 \end{numitem1}
Действительно, если $e=mw$ -- матчинговое ребро в ротации $C$, где $m\in I$, то при трансформации в $C$ новым матчинговым ребром в $\delta(m)$ становится активное ребро $a=mv$ в $C$; при этом выполняется $e<_m a$. При дальнейшем движении вдоль трассы матчинговые ребра в $\delta(m)$ могут перемещаться только ``вправо'' (по аналогичным причинам), и возврата к ребру $e$ не происходит. 
 
В заключение этого раздела скажем о компонентах графа объединения стабильных матчингов $\Sigma_G$. Для этого обозначим через $M^0$ множество ребер минимального матчинга $\Mmin$, принадлежащих древесным компонентам в  $\Gamma(\Mmin)$, и обозначим через $U_G$ подграф в $G$, являющийся объединением всех ротаций в $\Rscr_G$. Как указано в Следствии ~\ref{cor:min-max}(ii), $M^0$ принадлежит всем матчингам в $\Mscr(G)$ (поскольку активный граф $\Gamma(\Mmin)$ не имеет понижающих циклов, и учитывая Предложение~\ref{pr:tree-comp}); отсюда получаем, что каждое ребро в $M^0$ образует компоненту в $\Sigma_G$. В то же время очевидно, что каждая ротация целиком лежит в $\Sigma_G$, и поэтому $U_G\subseteq \Sigma_G$. Более того, так как каждый стабильный матчинг получается из $\Mmin$ в результате последовательности ротационных трансформаций, то справедливо

\begin{corollary} \label{cor:SigmaG}
$\Sigma_G$ является объединением $U_G=\cup\{C\in \Rscr_G\}$ и $\Mmin$. Каждое ребро в $M^0$ и каждая компонента в $U_G$ является компонентой в $\Sigma_G$.
\end{corollary}

Таким образом, $\Sigma_G$ строится эффективно (за время $O(|V| |E|)$, ввиду сказанного в следующем разделе). А также каждое ребро в $M^0$ и каждая компонента в $U_G$ является компонентой в $\Sigma_G$. В связи с этим можно спросить, есть ли в $\Sigma_G$ другие компоненты? Если да, то ими могут быть только ребра в $\Mmin$, которые находятся в цикло-содержащей компонентах $K$ графа $\Gamma(\Mmin)$, но при этом не входят в объединение ротаций $U_K$. Вопрос о существовании таких компонентах открытый для нас.


\section{Идеалы посета ротаций и стабильные матчинги}  \label{sec:ideal}

Как было сказано выше, естественный частичный порядок $\prec$ на стабильных матчингах превращает множество $\Mscr(G)$ в дистрибутивную решетку (см. Предложение~\ref{pr:lattice}). Известно, что любая дистрибутивная решетка изоморфна решетке идеалов некоторого посета. Для решетки $(\Mscr(G),\prec)$ Ирвинг и Лейтер~\cite{IL} указали явную конструкцию. (Хотя~\cite{IL} и последующая работа~\cite{ILG} имеют дело с полным двудольным графом $K_{n,n}$, полученные там результаты без большого труда распространяются на произвольный двудольный $G$.) А именно, опираясь на ключевой факт о постоянстве множества ротаций, ассоциированных с трассами (см. Предложение~\ref{pr:C-tau}), а также на определенный  способ задания структуры посета на множестве ротаций $\Rscr_G$, было показано соответствие между стабильными матчингами и идеалами (замкнутыми множествами) в таком посете. Это представление для $\Mscr(G)$ имеет важные применения, в частности, позволяет эффективно решать задачи линейной оптимизации для $\Mscr(G)$ и установить труднорешаемость задачи вычисления числа $|\Mscr(G)|$, о чем будет сказано в разделах~\SEC{polyhedr} и~\SEC{addit}.

Начнем с построения упомянутого посета на $\Rscr_G$. Напомним, что для трассы $\tau=(M_0,M_1,\ldots,M_N)$ мы обозначаем  через $\Rscr(\tau)$ последовательность ротаций $C_1,\ldots,C_N$, где $M_i$ получается из $M_{i-1}$ заменой вдоль $C_i$. 
\medskip

\noindent\textbf{Определение.} Для ротаций $C,D\in\Rscr_G$ скажем, что $C$ \emph{предшествует} $D$, и обозначим это как $C\lessdot D$, если для \emph{каждой} трассы $\tau$ в $(\Mscr(G),\prec)$ ротация $C$ встречается в последовательности $\Rscr(\tau)$ раньше, чем $D$. Иными словами, трансформация матчинга в $D$ никогда не может произойти ранее, чем трансформация в $C$. 
\medskip

Это бинарное отношение антисимметричное и транзитивное и дает частичный порядок на $\Rscr_G$. 

Чтобы объяснить связь с матчингами, рассмотрим $M\in\Mscr(G)$. Пусть $\tau=(M_0,M_1,\ldots,M_N)$ -- трасса, содержащая $M$, скажем, $M=M_i$, и пусть $\Rscr(\tau)=(C_1,\ldots, C_N)$. Тогда $M$ получается из $M_0=\Mmin$ последовательностью трансформаций относительно ротаций $C_1,\ldots,C_i$. Из Предложения~\ref{pr:C-tau} легко следует, что 
 \begin{numitem1} \label{eq:RscrM}
(неупорядоченное) множество $\Rscr_M:=\{C_1,\ldots, C_i\}$ не зависит от трассы $\tau$, проходящей через $M$;  аналогично для множества $\Rscr^+_M:=\{C_{i+1},\ldots, C_N\}$, и для множества $\Rscr_{M,M'}:=\{C_{i+1},\ldots, C_j\}$, где $M'=M_j$ для $j\ge i$.
  \end{numitem1}

В соответствии с определением $\lessdot$ отношение $C_j\lessdot C_k$ может выполняться, только если $j<k$. Более того, $\Rscr_M$ образует \emph{замкнутое множество}, или \emph{идеал}, в посете $(\Rscr_G,\lessdot)$ (т.е. из $C\lessdot D$ и $D\in \Rscr_M$ следует $C\in \Rscr_M$). Это дает инъективное отображение $\beta$ множества $\Mscr(G)$ в множество идеалов в $(\Rscr_G,\lessdot)$. В действительности, $\beta$ является биекцией, как показано в~\cite[Th.~4.1]{IL}.

\begin{prop} \label{pr:ideal-match}
Отображение $M\mapsto \Rscr_M$ устанавливает изоморфизм между решеткой  стабильных матчингов $(\Mscr(G),\prec)$ и решеткой $\Iscr(G)$ идеалов в $(\Rscr_G,\lessdot)$. 
  \end{prop} 
 \begin{proof} ~Рассмотрим стабильные матчинги $M$ и $L$, и пусть $E'$ и $E''$ -- множества ребер, принадлежащих циклам в $\Cscr^+(M,L)$ и $\Cscr^-(M,L)$, соответственно. Очевидно, подграфы, порожденные этими множествами, не пересекаются.  Из доказательства Предложения~\ref{pr:trass} следует, что $E'$ является объединением множества $R'$ ротаций, требуемых для получения $M\vee L$  из  $M$ (или, эквивалентно, $L$ из $M\wedge L$), а $E''$ -- объединением множества $R''$ ротаций, требуемых для получения $M\vee L$  из  $L$ (эквивалентно, $M$ из $M\wedge L$). Следовательно, ротации $R'$ и $R''$ ``перестановочны'': чтобы получить матчинг $M\vee L$ из $M\wedge L$ можно применять сначала  $R'$, а затем $R''$, или наоборот. Это дает $\beta(M)\cap \beta(L)=\beta(M\wedge L)$ и $\beta(M)\cup \beta(L)=\beta(M\vee L)$ (т.е. $\beta$ определяет гомоморфизм решеток). Теперь утверждение легко следует из того факта, что для ротаций $C,D\in\Rscr_G$ не выполняется $C\lessdot D$ тогда и только тогда, когда найдется стабильный матчинг $M$ такой, что $D\in\Rscr_M\not\ni C$.
\end{proof}

Как следствие, $\Mscr(G)$ биективно множеству $\Ascr(G)$ анти-цепей в $(\Rscr_G,\lessdot)$. (Напомним, что \emph{анти-цепь} в посете $(P,<)$ -- это максимальное по включению множество $A$ попарно несравнимых элементов. Оно определяет идеал $\{p\in P\colon p\le a~ \exists a\in A\}$, и определяется идеалом как множеством его максимальных элементов.)
\smallskip

Мы завершаем этот раздел, объясняя как эффективно проверять отношение $\lessdot$. Заметим, что $C,D$ не связаны этим отношением, если они являются ротациями для одного и того же стабильного матчинга $M$ (и, следовательно, трансформации в $C$ и $D$ могут производиться в произвольном порядке); в этом случае $C$ и $D$ лежат в разных компонентах в $\Gamma(M)$, и $V_C\cap V_D=\emptyset$. С другой стороны, если $C$ и $D$ имеют общую вершину $v$, то они сравнимы по $\lessdot$.

Действительно, для $C$ и $D$ матчинговые ребра, инцидентные $v$, различные, и  порядок следования $C$ и $D$ для любой трассы один и тот же и определяется порядком $<_v$ для этих ребер (см. объяснение свойства~\refeq{different}).  Вообще, множество ротаций $\Rscr^v$, содержащих фиксированную вершину $v$, упорядочено по $\lessdot$  (скажем, $\Rscr^v=(C_1\lessdot C_2\lessdot\cdots \lessdot C_k)$), и этот порядок легко определяется, так как он согласуется с порядком в $<_v$, если $v\in I$, и обратным порядком, если $v\in J$. 

Есть еще одна локальная причина того, что две ротации $C,D$ должны быть сравнимы по $\lessdot$. А именно,
  \begin{numitem1} \label{eq:second_reas}
предположим,  $C$ содержит матчинговое ребро $e=mw$ и активное ребро $vw$, и $D$ содержит матчинговое ребро $e'=m'w'$, для которых: $b=m'w$ -- ребро в $G$, не принадлежащие никаким ротациям и являющимся первым после $e'$ ребром в $\delta(m')$, и при этом выполняется: $a<_w b<_w e$; тогда $C\lessdot D$.
  \end{numitem1}
  
Действительно, ввиду $a<_wb<_w e$, после трансформации в  $C$ ребро $b$ становится недопустимым. С другой стороны, если бы $D$ предшествовало $C$ в какой-то трассе, то именно $b$ было бы активным ребром в $\delta(m')$ и вошло бы в ротацию. 

Пример в разд.~\SEC{rotat} именно таков: здесь после замены вдоль ротации $C$ ребро $e=2v$ становится недопустимым, ввиду $d'<_v e<_v a'$ и $d,d'\in M_2$ (при этом $e$ -- первое после $d$ ребро в $\delta(2)$). В этом Примере посет состоит из двух ротаций $C,D$ при соотношении $C\lessdot D$, и есть ровно три идеала: $\emptyset$, $\{C\}$ и $\{C,D\}$, отвечающие стабильным матчингам $M_1=\Mmin$, $M_2$ и $M_3=\Mmax$ (а $\{D\}$ -- не идеал).
\smallskip

Далее, мы знаем, что для нахождения множества ротаций $\Rscr_G$ достаточно определить $\Mmin$ и затем построить произвольную трассу, делая трансформации на ротациях в текущих матчингах. (Это выполняется естественным алгоритмом за время $O(|V| |E|)$. В~\cite{gus} предложен алгоритм формирования $\Rscr_G$ за время $O(|V|^2)$.) 

В свою очередь в~\cite[Sec.~5]{ILG} был установлен важный фокт, что для эффективного описания порядка $\lessdot$, действительно достаточно рассмотреть пересечения ротаций и учесть~\refeq{second_reas}. А именно, образуем ориентированный граф $H=(\Rscr_G,\Escr)$, ребрами которого являются пары ротаций $(C,D)$ такие, что либо $V_C\cap V_D\ne\emptyset$ и $C\lessdot D$, либо выполняется~\refeq{second_reas} (этот  граф легко строится за время $O(|V| |E|)$; с некоторой изобретательностью время можно уменьшить до $O(|E|)$). Можно проверить, что

 \begin{numitem1} \label{eq:conseqv}
граф $H$ содержит ребро $(C,D)$ тогда и только тогда, когда найдется стабильный матчинг $M$ такой, что $M$ допускает ротацию $C$, но не $D$, а после трансформации $M$ относительно $C$ полученный стабильный матчинг $M'$ уже допускает $D$.
  \end{numitem1}

Граф $H$ ациклический, и если какая-либо вершина $D$ в $H$  достижима ориентированным путем из $C$, то выполняется $C\lessdot D$. Верно и обратное.

\begin{prop} \label{pr:lessdot}
Если $C\lessdot D$, то в $H$ имеется ориентированный путь из $C$ в $D$. Следовательно, отношение достижимости для вершин в $H$ совпадает с $\lessdot$.
 \end{prop}
 
 \begin{proof}
Пусть $\Iscr'$ -- множество идеалов в посете для $H$; тогда $\Iscr'\supseteq\Iscr(G)$. Надо доказать, что $\Iscr'=\Iscr(G)$. Предположим, это не так, т.е. имеется $X\in \Iscr'$, не являющийся идеалом в $(\Rscr_G,\lessdot)$. Пусть при этом $X$ выбрано так, чтобы число элементов $|X|$ было максимальным. Также выберем идеал $R\in \Iscr(G)$ такой, что $R\subset X$, и при этом $|R|$ максимально.
Согласно Предложению~\ref{pr:ideal-match}, $R=\Rscr_M$ для некоторого стабильного матчинга $M$. Возьмем какую-либо ротацию $С$ для $M$. Ввиду максимальности $R$, элемент $C$ должен быть вне $X$. Тогда $R':=R\cup\{C\}$ -- идеал в $(\Rscr_G,\lessdot)$ (соответствующий матчингу $M'$, получаемому из $M$ заменой вдоль $C$). Также $X':=X\cup \{C\}\in \Iscr'$. 

Ввиду максимальности $X$, имеем $X'\in \Iscr(G)$. Поскольку $X'\supset R'$, и оба $X',R'$ -- идеалы в $(\Rscr_G,\lessdot)$, множество $X'-R'$ содержит элемент (ротацию) $D$ такой, что $R'':=R'\cup \{D\}\in\Iscr(G)$. В то же время $X'':=R\cup\{D\}$ не является идеалом в $(\Rscr_G,\lessdot)$; иначе, ввиду $R\subset X''\subseteq X$, мы получили бы противоречие с максимальностью $R$ (если $X''\ne X$) или с условием на $X$ (если $X''=X$).

Итак, мы пришли к ситуации, когда $R,R',R''\in\Iscr(G)$, $R'=R\cup\{C\}$, $R''=R'\cup\{D\}$, но $R\cup \{D\}\notin \Iscr(G)$. Тогда ротация $D$ не имеет место для матчинга $M$, но появляется сразу после того, как в $M$ производится замена вдоль ротации $C$. Это в точности означает, что граф $H$ содержит ребро $(C,D)$; ср.~\refeq{conseqv}. Но такое ребро идет из $\Rscr_G-X$ в $X$, вопреки тому, что $X$ принадлежит $\Iscr'$. 
 \end{proof}
 
 \noindent\textbf{Замечание} Пусть нам даны стабильные матчинги $M$ и $M'$ в двудольном $G$ такие, что $M\prec M'$, и пусть $[M,M']$ обозначает множество (``интервал'') стабильных матчингов $L$, для которых $M\preceq L\preceq M'$. (Как специальные случаи можно рассматривать $M=\Mmin$ или $M'=\Mmax$.)  Нас интересует возможность ``очистки'' графа $G$ путем удаления части его ребер с тем, чтобы для полученного графа $G'$ множество стабильных матчингов в точности совпадало с интервалом $[M,M']$ для $G$. Чтобы попытаться ответить на данный вопрос, возьмем произвольную трассу $\tau$, проходящую через оба $M,M'$, и рассмотрим множество ротаций, используемых на участке трассы от $M$ до $M'$, т.е. $\Rscr_{M,M'}$, см.~\refeq{RscrM}.  (Трасса $\tau$ и множество $\Rscr_{M,M'}$ строятся эффективно; ср. доказательство Предложения~\ref{pr:trass}.) Тогда интервал $[M,M']$ состоит ровно из тех $L\in\Mscr(G)$, которые получаются из $M$ применением последовательности ротаций из множества $\Rscr_{M,M'}$. Поэтому искомый граф $G'$ должен включать $M$ и объединение всех ротаций из $\Rscr_{M,M'}$. Однако он также должен учитывать структуру графа $H$ (см.~\refeq{conseqv}), чтобы избежать появления ``лишних'' матчингов, использующих ротации из $\Rscr_{M,M'}$. Следовательно, нужно добавить ребра как в~\refeq{second_reas} (вида $m'w$), чтобы гарантировать сохранение указанного там отношения $C\lessdot D$ для соответствующих ротаций $C,D$ в $\Rscr_{M,M'}$. Гипотеза: подграф $G'$ существует и получается из $G$ удалением всех остальных ребер.

 
\section{Оптимальные стабильные матчинги}  \label{sec:optimal}

В этом разделе мы рассматриваем ситуацию, когда двудольный граф $G=(V,E,<)$ дополнительно снабжен вещественной функцией \emph{весов} (или \emph{стоимостей}) ребер $c:E\to \Rset$. Нас интересует \emph{задача о стабильном матчинге минимального веса}:

\begin{numitem1} \label{eq:min-cost}
Найти стабильный матчинг $M\in\Mscr(G)$, минимизирующий общий вес $c(M):=\sum_{e\in M} c(e)$.
\end{numitem1}

(Ввиду того, что все стабильные матчинги в $G$ имеют одинаковый размер, добавление константы к функции $c$ фактически не меняет задачи; поэтому можно считать функцию $c$ неотрицательной. При замене $c$ на $-c$ получаем задачу максимизации веса $c(M)$.)

В важном частном случае этой задачи, известном как задача об \emph{оптимальном стабильном матчинге} (см.~\cite{MW71}), вес $c(e)$ ребра $e=mw$ определяется как
\begin{equation} \label{eq:opt-weight}
c(e):= r_I(e)+r_J(e),
 \end{equation}
где $r_I(e)$ -- порядковый номер ребра $e$ в упорядоченном (согласно $<_m$) списке $\delta(m)$, и аналогично, $r_J(e)$ -- номер в списке $\delta(w)$. (Такая постановка широко обсуждалась в литературе, в частности, в~\cite{knu,pol}. Здесь интересуются матчингом, который ``наиболее благоприятен по суммарным (или средним) предпочтениям для всех персон'' (в то время как ``алгоритм отложенных принятий'' (АОП) дает наилучшее решение только для одного множества из $I$ и $J$).) В~\cite{ILG} задача с весовой функцией как в~\refeq{opt-weight} также называется задачей об \emph{эгалитарном стабильном матчинге}.

Задача~\refeq{min-cost} может быть сформулирована как задача линейного программирования с $0,\pm 1$-матрицей размера $(|V|+|E|)\times |E|$ (о чем будет сказано в следующем разделе), поэтому она может быть решена универсальным сильно полиномиальным алгоритмом (усиленной версией метода эллипсоидов в~\cite{tar}). Однако изложенное в предыдущем разделе представление $\Mscr(G)$ в виде множества идеалов в посете ротаций $(\Rscr_G,\lessdot)$ дает возможность решать~\refeq{min-cost} намного более экономным и наглядным методом. Это было предложено Ирвингом-Лейтером-Гасфилдом в~\cite{ILG}, и ниже мы описываем этот метод (применительно к произвольному двудольному графу $G$).

Для заданной функции весов $c$ и произвольной ротации $R=e_1a_1e_2a_2\cdots e_ka_k$ в $\Rscr_G$, где ребра $e_i$ матчинговые, а ребра $a_i$ активные, определим вес $R$ как

\begin{equation} \label{eq:rot-weight}
 c^R:=\sum(c(a_i)-c(e_i) \colon i=1,\ldots,k).
  \end{equation}
  
При трансформации относительно ротации $R$ к весу текущего матчинга прибавляются веса активных ребер и из него вычитаются веса матчинговых ребер в $R$. Поэтому вес любого стабильного матчинга $M$ выражается как
  $$
  c(M)=c(\Mmin)+ \sum(c^R\colon R\in\Rscr_M),
  $$
где $\Rscr_M$ -- идеал в $(\Rscr_G,\lessdot)$, соответствующий $M$ (т.е. множество ротаций, возникающих на участке (любой) трассы от $\Mmin$ до $M$).

Таким образом, \refeq{min-cost} эквивалентно задаче нахождения идеала минимального веса. В общем случае задача такого рода выглядит следующим образом:

\begin{numitem1} \label{eq:min-cl-set}
Для ориентированного графа $Q=(V_Q,E_Q)$ и функции $\zeta: V_Q\to\Rset$ найти замкнутое множество $X\subseteq V_Q$ минимального веса $\zeta(X):=\sum(\zeta(v)\colon v\in X)$.
  \end{numitem1}
(Напомним, что множество $X$ замкнутое, если нет ребер, идущих из $V_Q-X$ в $X$. В частности, замкнутыми множествами являются $\emptyset$ и $V_Q$. Без ограничения общности можно считать граф $Q$ ациклическим, так как любой ориентированный цикл не может разделяться замкнутым множеством, и его можно стянуть в одну вершину суммарного веса.)

Решение задачи~\refeq{min-cl-set}, предложенное Пикаром~\cite{pic}, состоит в ее сведении к  задаче о минимальном разрезе в ориентированном графе $\hat Q=(\hat V,\hat E)$ с пропускными способностями ребер $h(e)$, $e\in \hat E$, определяемыми следующим образом.

Положим $V^+:=\{v\in V_Q\colon \zeta(v)>0\}$ и  $V^-:=\{v\in V_Q\colon \zeta(v)<0\}$. Граф $\hat Q$ получается из $Q$ добавлением двух вершин: ``источника'' $s$ и ``стока'' $t$, а также множества ребер $E^+$ из $s$ в $v$ для всех $v\in V^+$ и множества ребер $E^-$ из $u$ в $t$ для всех $u\in V^-$. Пропускные способности ребер $e\in \hat E$ задаются так:
    $$
    h(e):=\left\{
    \begin{array}{rcl}
    \zeta(v), && \mbox{если $e=sv\in E^+$}; \\
    |\zeta(u)|, && \mbox{если $e=ut\in E^-$}; \\
    \infty, && \mbox{если $e\in E_Q$}.
     \end{array}
   \right.
   $$

Для подмножества $S\subseteq\hat V$ такого, что $s\in X\not\ni t$, обозначим через $\delta(S)$ множество ребер в $\hat Q$, идущих из $S$ в $\hat V-S$, называемое $s$--$t$ \emph{разрезом}; величина $h(\delta(S)):=\sum(h(e)\colon e\in\delta(S))$ считается пропускной способностью этого разреза. 

 \begin{lemma} {\rm \cite{pic}} \label{lm:min-cut}
Подмножество $X\subseteq V_Q$ является замкнутым множеством минимального веса в $(Q,\zeta)$ тогда и только тогда, когда $\delta((V_Q-X)\cup\{s\})$ является $s$--$t$ разрезом минимальной пропускной способности в $(\hat Q,h)$.
  \end{lemma}
    \begin{proof}
Заметим, что $X\subseteq V_Q$ -- замкнутое множество тогда и только тогда, когда разрез $E'=\delta((V_Q-X)\cup\{s\})$ не содержит ребер бесконечной пропускной способности; эквивалентно, $E'$ содержится в $E^+\cup E^-$. Для такого разреза, состоящего из ребер вида $sv$ для $v\in X$ и ребер вида $ut$ для $u\in V_Q-X$, пропускная способность выглядит так:
 \begin{multline*}
 h(E')=\zeta(X\cap V^+) + \sum(|\zeta(u)|\colon u\in(V_Q-X)\cap V^-) \\
   =\zeta(X\cap V^+)+\zeta(X\cap V^-)-\zeta(V^-)=\zeta(X)-\zeta(V^-).
   \end{multline*}

Следовательно, вес замкнутого множества отличается от пропускной способности соответствующего разреза на постоянную величину $-\zeta(V^-)$, откуда получаем требуемое утверждение.    
  \end{proof}

Таким образом, задача~\refeq{min-cost} сводится к задаче о минимильном двухполюсном разрезе в сети с $N=O(|E|)$ вершинами и $A=O(|E|)$ ребрами (учитывая то, что вместо всего посета $(\Rscr_G,\lessdot)$ достаточно рассматривать порождающий граф $H$, имеющий $(O(|E|)$ ребер (см. Предложение~\ref{pr:lessdot}). Применяя быстрые алгоритмы для задачи о максимальном потоке и минимальном разрезе (см., например, обзор в~\cite[Sec.~10.8]{sch}), можно получить временную оценку $O(NA\log N)\simeq O(|V|^4 \log |V|)$. В~\cite{ILG} приведена имплементация, решающая~\refeq{min-cost} за время $O(|V|^4)$.
 \medskip
 
\noindent\textbf{Замечание.} В~\cite{karz84} доказывалась труднорешаемость некоторых вариантов задачи о замкнутых множествах. Используя один из них, а также тот факт, что любой транзитивно замкнутый граф реализуется как посет ротаций (о чем будет сказано в разделе~\SEC{addit}), можно показать NP-трудность следующего усиления задачи~\refeq{min-cost}: для двудольного $G=(V,E,<)$, функций $c,g: E\to\Rset_+$ и числа $K\in\Rset_+$ найти стабильный матчинг $M\in\Mscr(G)$, минимизирующий $c(M)$ при условии $g(M)\ge K$. (Здесь $g(M)$ можно интерпретировать как прибыль, а $c(M)$ как затраты при организации союзов (или контрактов) в $M$.)


\section{Полиэдральные аспекты и медианные стабильные матчинги}  \label{sec:polyhedr}

Как мы упоминали ранее, для двудольного графа $G=(V=I\sqcup J, E,<)$ имеется полиэдральная характеризация многогранника стабильных матчингов $\Pscrst(G)$, задаваемая линейным от $|V|,|E|$ числом неравенств. Здесь $\Pscrst(G)$ -- выпуклая оболочка множества характеристических векторов $\chi^M$ стабильных матчингов $M$ в пространстве $\Rset^{E}$. Первоначальное описание $\Pscrst(G)$ (в случае $G\simeq K_{n,n}$) было дано в работе Ванде Вейта~\cite{VV} (и повторена в работе~\cite{rot}). Ниже мы приводим описание (несколько отличающееся по виду, но близкое к тому, что в~\cite{VV}) и доказательство, основываясь на изложении в~\cite[Sec.~18.5g]{sch}. 

Для ребер $e,f\in E$ будем писать $f\prec e$, если они имеют общую вершину $v$, и выполняется $f<_v e$ (т.е. $f$ предпочтительнее, чем $e$). Для $e\in E$ обозначим $\gamma(e)$ множество ребер $f$ таких, что $f\preceq e$ (в частности, $e\in\gamma(e)$). 

Напомним, что многогранник матчингов в двудольном случае описывается системой линейных неравенств
 \begin{eqnarray} 
 x(e)\ge 0, && e\in E; \label{eq:xge0}\\
 x(\delta(v))\le 1, && v\in V. \label{eq:xdeltav}
   \end{eqnarray}
   
В случае стабильных матчингов добавляется еще один тип неравенств:
  \begin{equation} \label{eq:gammae}
  x(\gamma(e))\ge 1, \qquad e\in E. 
  \end{equation}
  
\begin{prop} \label{pr:st-match-pol}
Система~\refeq{xge0}--\refeq{gammae} описывает в точности множество векторов $x\in \Rset^E$, принадлежащих  $\Pscrst(G)$.
 \end{prop}
 \begin{proof}
~Легко видеть, что для любого стабильного матчинга $M$ вектор $x=\chi^M$ удовлетворяет~\refeq{xge0}--\refeq{gammae}. Поэтому достаточно доказать, что если $x$ -- вершина многогранника $\Pscr'$, определяемого~\refeq{xge0}--\refeq{gammae}, то вектор $x$  целочисленный. 

Положим $E^+:=\{e\in E\colon x(e)>0\}$ и обозначим $V^+$ множество вершин в $G$, покрытых $E^+$. Для $v\in V^+$ обозначим $e_v$ наилучшее ребро в $\delta(v)\cap E^+$ относительно порядка $<_v$. Справедливо следующее свойство:
  \begin{numitem1} \label{eq:worst}
для $v\in V^+$ и $e_v=\{v,u\}$ ребро $e_v$ является наихудшим в $(\delta(u)\cap E^+,<_u)$; кроме того, выполняется $x(\delta(u))=1$.
  \end{numitem1}
  
Действительно, обозначая $e:=e_v$, имеем
  \begin{gather*}
  1\le \sum(x(f)\colon f\preceq e)\; \mbox{(ввиду~\refeq{gammae})}\; 
  =\sum(x(f)\colon f\le_u e)\; \mbox{(ввиду определения $e_v$)} \\
  =x(\delta(u))-\sum(x(f)\colon f>_u e)\;\le 1-\sum(x(f)\colon f>_u e)\; \mbox{(применяя~\refeq{xdeltav} к $u$)}.
   \end{gather*}
 
Здесь все неравенства должны обращаться в равенства. Это дает $\sum(x(f)\colon f>_u e)=0$ и $x(\delta(u))=1$, откуда следует~\refeq{worst}.

Образуем множества $M:=\{e\in E^+\colon e=e_v$ для $v\in I\}$ и $L:=\{e\in E^+\colon e=e_v$ для $v\in J\}$. Для любой вершины $v\in I\cap V^+$ наилучшее ребро в $\delta(v)\cap E^+$ принадлежит $M$, и наихудшее ребро, и только оно, принадлежит $L$ (ввиду~\refeq{worst}); для вершин в $J\cap V^+$ поведение обратное. Отсюда следует, что оба $M$ и $L$ являются матчингами. Каждое ребро $e\in M\cap L$ образует компоненту в подграфе $(V^+, E^+)$; это дает  $x(e)=1$ (ввиду~\refeq{xdeltav},\refeq{gammae}). В частности, $x$ целочисленный, если $M=L$.  

Пусть теперь $M\ne L$. Очевидно, для любого ребра $e$ в $M':=M-L$ или в $L':=L-M$ выполняется $0<x(e)<1$. Поэтому можно выбрать достаточно малое число $\eps>0$ так, чтобы вектора $x':=x+\eps\chi^{M'}-\eps\chi^{L'}$ удовлетворяли~\refeq{xge0} и~\refeq{xdeltav}. Проверим, что оба $x'$ и $x''$ удовлетворяют также и~\refeq{gammae}.
 
Для этого рассмотрим ребро $e=mw$ с $x(e)<1$ и предположим, что $x'(\gamma(e))<x(\gamma(e))$. Это возможно только если $a\le_w e<_w b$, где $\{a\}= L\cap\delta(w)$ и $\{b\}= M\cap \delta(w)$. При этом должно выполняться либо (i) $m\notin V^+$, либо (ii) $m\in V^+$, и $e$ более предпочтительное, чем ребро в $M\cap\delta(m)$. Но в этих случаях из равенства $x(\delta(w)=1$ (ввиду~\refeq{worst}) и неравенства $x(b)>0$ следует
   $$
   x(\gamma(e))=\sum(x(f)\colon f\le_w e)\le x(\delta(w))-x(b)<1,
   $$
что невозможно. Аналогично,~\refeq{gammae} верно для $x''$.

Таким образом, $x',x''\in\Pscr'$. Но $x'\ne x''$ и $(x'+x'')/2=x$. Это противоречит тому, что $x$ -- вершина в $\Pscr'$. 
 \end{proof}

Имеется еще одно полезное свойство, показанное в~\cite{RRV}; мы приводим его в несколько иной, но эквивалентной, форме.

\begin{lemma} \label{lm:x(e)>0}
~Пусть $x\in \Pscrst(G)$, и пусть $e=mw$ -- ребро в $G$, для которого $x(e)>0$. Тогда $x(\delta(m))=x(\delta(w))=x(\gamma(e))=1$.
  \end{lemma}
  \begin{proof} 
~Представим $x$ как $\alpha_1\chi^{M_1}+\cdots+ \alpha_k\chi^{M_k}$, где  $M_i$ -- стабильный матчинг, $\alpha_i>0$, и $\alpha_1+\cdots+ \alpha_k=1$. Обозначим $x_i:=\chi^{M_i}$. Ввиду $x(e)>0$, ребро $e$ принадлежит некоторому $M_i$. Тогда $x_i(\delta(m))=x_i(\delta(w))=x_i(\gamma(e))=1$. Аналогичные равенства имеют место и для матчинга $M_j$, не содержащего $e$. Действительно, так как $M_i$ и $M_j$ покрывают одно и то же множество вершин (ввиду Предложения~\ref{pr:cover_vert}), в $M_j$ есть ребро $e'$, инцидентное $m$, и ребро $e''$, инцидентное $w$. Более того, рассматривая пару $M_i,M_j$ и применяя~\refeq{3path}, получаем либо $e'<_m e<_w e''$, либо $e'>_m e>_w e''$. В обоих случаях в $\gamma(e)$ попадает ровно одно ребро из $e',e''$; поэтому $x_j(\gamma(e))=1$. Теперь утверждение следует из $\alpha_1+\cdots +\alpha_k=1$.
\end{proof}

Эта лемма помогает получить интересный результат Тео и Сетарамана.

\begin{prop} {\rm\cite{TS}} \label{pr:k-l-k}
Пусть $M_1,\ldots,M_\ell$ -- стабильные матчинги в $G$. Для каждого $m\in \tilde I$ обозначим $E_m$ список (с возможными повторениями) ребер в $\delta(m)$, принадлежащих этим матчингам, упорядоченный согласно $<_m$, и пусть $e_m(i)$ обозначает $i$-й элемент в $E_m$, $i=1,\ldots,\ell$. Аналогично, для каждого $w\in\tilde J$  пусть $e_w(j)$ обозначает $j$-й элемент в списке $E_w$ ребер в $\delta(w)$, принадлежащих данным матчингам,  упорядоченном согласно $<_w$, $j=1,\ldots,\ell$. Тогда для любого $k\in\{1,\ldots,\ell\}$ множество ребер $A(k):=\{e_m(k)\colon m\in\tilde I\}$ совпадает с множеством ребер $B(\ell-k+1):=\{e_w(\ell-k+1)\colon w\in \tilde J\}$ и образует стабильный матчинг в $G$.
 \end{prop}
\begin{proof}
~Положим $x_i:=\frac{1}{\ell} \chi^{M_i}$. Тогда $x=x_1+\cdots+x_\ell$ лежит в многограннике $\Pscrst(G)$ (является ``дробным стабильным матчингом'').

Вначале предположим для простоты, что все ребра в $M_1,\ldots,M_\ell$ различны. Тогда для $m\in\tilde I$ список $E_m$ состоит из $\ell$ различных ребер, и ребро $e_m(k)$ является $k$-м элементом в $E_m$. Из Леммы~\ref{lm:x(e)>0}, примененной к $x$ и $e=e_m(k)$ следует, что множество $\gamma(e)$ состоит из $\ell$ элементов. Тогда $e$ является в точности $(\ell-k+1)$-м элементом в списке $E_w$ (по предположению состоящем из $\ell$ различных ребер), и тем самым $e_m(k)=e_w(\ell-k+1)$.

Таким образом, $A(k)=B(\ell-k+1)$, откуда следует, что $A(k)$ является матчингом в $G$. Чтобы показать стабильность $A(k)$ рассмотрим произвольное ребро $e=mw\notin A(k)$ в $G$. Надо проверить, что $e$ не блокирует $A(k)$, или, эквивалентно, что $A(k)\cap\gamma(e)\ne \emptyset$.

По крайней мере один из концов $e$, скажем, $m$, принадлежит $\tilde V$ (и покрывается $A(k)$); иначе для $x$ и $e$ мы имели бы $x(\gamma(e))=0$ вопреки~\refeq{gammae}. Если $w\notin \tilde V$, то $x(\delta(w))=0$, и из неравенства $x(\gamma(e))\ge 1$ следует, что для всех ребер $f\in E_m$ (включая $f=e_m(k)$) выполняется $f<_m e$. Это дает требуемое $A(k)\cap \gamma(e)\ne\emptyset$.

Пусть теперь $w\in\tilde V$. Ввиду $x(\gamma(e))\ge 1$, число ребер $f$ в $E_m\cup E_w$, для которых $f\preceq e$, не меньше $\ell$. Тогда должно выполняться по крайней мере одно из $e_m(k)\le_m e$ и $e_w(\ell-k+1)\le_w e$, откуда следует $A(k)\cap\gamma(e)\ne\emptyset$.

В том случае, когда в $M_1,\ldots,M_\ell$ есть общие ребра, можно рассмотреть мультиграф, получаемый из $G$ заменой каждого ребра $e=mw$ с $x(e)>0$ на $\ell x(e)=:r$ параллельных ребер $e^1,\ldots,e^r$. При этом продолжение порядка $<_m$ на эти ребра назначается противоположным продолжению порядка $<_w$, скажем, $e^1<_m \cdots <_m e^r$ и $e^1>_w \cdots >_w e^r$. Требуемое утверждение для этого общего случая получается повторением (с незначительными уточнениями) рассуждений для случая непересекающихся матчингов выше.
\end{proof}
   
\begin{corollary} \label{cor:median}
Если $\ell$ нечетно, то для $k:=(\ell+1)/2$ множество, состоящее из $k$-х по порядку элементов в списках ребер $E_v$, инцидентных $v$ и принадлежащих $M_1,\ldots,M_\ell$, для всех вершин $v\in \tilde V$, является стабильным матчингом.
  \end{corollary}
  
Такой матчинг называют \emph{медианным} для $M_1,\ldots,M_\ell$. В~\cite{TS} был поставлен вопрос о возможности эффективного нахождения медианного стабильного матчинга среди всех стабильных матчингов в $G$ (или ``почти медианного'', когда число стабильных матчингов четное). Нам это представляется маловероятным, в свете того, что задача вычисления числа $|\Mscr(G)|$ является труднорешаемой, о чем будет сказано ниже.


\section{Дополнительные замечания}  \label{sec:addit}

Кнут~\cite{knu} указал примеры, когда число стабильных матчингов двудольного графа экспоненциально велико по сравнению с размером графа, и задал вопрос о сложности точного вычисления этого числа. Ответ был дан Ирвингом и Лейтером~\cite{IL}, которые показали труднорешаемость такой задачи (рассматривая графы вида $(K_{n,n},<)$). Ниже мы кратко поясним идею их доказательства.

Напомним некоторые понятия (отсылая за точными определениями к~\cite{val} или~\cite[Разд.~7.3]{GJ}). Не вдаваясь в полную логическую строгость, можно понимать, что описание той или иной перечислительной задачи (enumeration problem) $\Pscr$ состоит из бесконечного семейства $\Sscr$ конечных множеств, и для каждого $S\in\Sscr$  имеется семейство $\Fscr(S)$ подмножеств в $S$ (``объектов''). В задаче $\Pscr$ требуется для заданного $S\in \Sscr$ определить число $|\Fscr(S)|$. Говорят, что $\Pscr$ является $\# P$-\emph{задачей} (или $KP$-\emph{задачей}, в терминологии некоторых авторов), если распознавание объекта проводится за полиномиальное время, т.е. имеется алгоритм, который для любых $S\in\Sscr$ и $X\subseteq S$ за время, полиномиальное от $|S|$, определяет принадлежит или нет данное множество $X$ семейству $\Fscr(S)$. Говорят, что $\#P$-задача $\Pscr=\{\Fscr(S), S\in\Sscr\}$ является $\#P$-\emph{полной} (или универсальной в классе $\#P$), если любая другая $\#P$-задача $\Pscr'=\{\Fscr'(S'), S'\in\Sscr'\}$ сводится к ней за полиномиальное время (т.е. есть отображение $\omega:\Sscr'\to \Sscr$ такое, что для каждого $S'\in\Sscr$ число $|\Fscr'(S')|$ определяется из $|\Fscr(\omega(S'))|$ за время, полиномиальное от $|S'|$).

В рассматриваемой нами задаче роль cемейства $\Sscr$ играет совокупность реберных множеств $E$ двудольных графов $G=(V,E,<)$, а роль семейства $\Fscr(S)$, $S\in\Sscr$, -- соответствующее множество стабильных матчингов в $G$. Задача определения $|\Mscr(G)|$ -- это действительно $\#P$-задача, так как для любого подмножества $M\subseteq E$ можно определить, является ли $M$ стабильным матчингом, за время $O(|E|)$. 

Заметим, что $\#P$-аналог любой $NP$-полной задачи является ``труднорешаемым''
(поскольку в последней  задаче требуется ``всего лишь'' определить, является ли непустым соответствующее семейство объектов $\Fscr(S)$ (например, содержит ли заданный граф хотя бы один гамильтонов цикл), а в первой нужно найти количество объектов). Однако есть $P$-задачи, чьи перечислительные аналоги являются $\#P$-полными. Именно таковой и является задача определения $|\Mscr(G)|$. Это вытекает из следующих двух результатов.

 \begin{prop} \label{pr:poset-anti}
Задача определения числа анти-цепей в конечном посете является $\#P$-полной.
 \end{prop}
\begin{prop} \label{pr:poset-match}
Пусть $(P,<')$ -- посет на $n$ элементах. Существует и может быть построен за время полиномиальное от $n$ двудольный граф \mbox{$G=(V,E,<)$} такой, что его ротационный посет $(\Rscr_G,\lessdot)$ изоморфен $(P,<')$. Следовательно (ввиду Предложения~\ref{pr:ideal-match}), число $|\Mscr(G)|$ стабильных матчингов в $G$ равно числу анти-цепей (или числу идеалов) в посете $(P,<')$.
 \end{prop}
 
Предложение~\ref{pr:poset-anti} было установлено Прованом и Боллом~\cite{PB}. Предложение~\ref{pr:poset-match} было доказано в~\cite[Sec.~5]{IL} путем явной конструкции требуемого графа $G$ для заданного посета~$(P,<')$.


\begin{thebibliography}{99}
\small
 %
\bibitem{BB} M.~Baiou and M.~Balinski, Erratum: the stable allocation (or ordinal transportation) problem. \textsl{Math. Oper. Res.} \textbf{27} (4) (2002) 662--680.
%
\bibitem{DM} B.C.~Dean and S.~Munshi, Faster algorithms for stable allocation problems.  \textsl{Algorithmica} \textbf{58} (1) (2010) 59--81.
  %
\bibitem{GS} D.~Gale and L.S.~Shapley, College admissions and the stability of marriage. \textsl{Amer. Math. Monthly} \textbf{69} (1) (1962) 9--15.
 %
 \bibitem{gus} D. Gusfield, Three fast algorithms for four problems in stable marriage. \textsl{SIAM J. Comput.} \textbf{16} (1987) 111--128. 
 %
 \bibitem{GJ} М. Гэри, Д. Джонсон, \textsl{Вычислительные машины и труднорешаемые задачи}, М.: Мир, 1982
 %
\bibitem{irv} R.W.~Irving, An efficient algorithm for the “stable roommates” problem. \textsl{J. Algorithms} \textbf{6} (1985) 577--595.
 %
\bibitem{IL}
R. W. Irving, P. Leather, The complexity of counting stable marriages. \textsl{SIAM J. Comput.} \textbf{15} (1986) 655--667. 
 %
\bibitem{ILG}
R. W. Irving, P. Leather, D. Gusfield, An efficient algorithm for the optimal stable marriage problem. \textsl{J. ACM} \textbf{34} (1987) 532--543.
%
  \Xcomment{
 \bibitem{karz} А.В.~Карзанов, Нахождение максимального потока в сети методом предпотоков. \textsl{Доклады АН СССР} \textbf{215} (1974) 49--52. 
(English translation in \textsl{Soviet Math. Dokl.} \textbf{15} (2) (1974)
434--437.)
 }
 \bibitem{karz84} А.В. Карзанов, О замкнутых множествах ориентированного графа. \textsl{ЖВМиМФ} \textbf{24} (1984) с.~1903--1906.
 %
\bibitem{knu} D.E. Knuth, Mariages stables. \textsl{Les Presses de l’Universit\'e de Montreal}, Montreal, 1976.
  %
\bibitem{MW70}
D. McVitie, L.B. Wilson, Stable marriage assignments for unequal sets. \textsl{BIT} \textbf{10} (1970) 295--309.
%
\bibitem{MW71}
D. McVitie, L.B. Wilson, The stable marriage problem. \textsl{Commun. ACM} \textbf{14} (1971) 486--492.
%
\bibitem{pic} J. Picard,  Maximum closure of a graph and applications to combinatorial problems. \textsl{Manage. Sci.} \textbf{22} (1976) 1268--1272.
 %
 \bibitem{pol} G. Polya, R.E. Tarjan, D.R. Woods. \textsl{Notes on Introductory Combinatorics}, Birkbauser Verlag, Boston, Mass., 1983.
  %
  \bibitem{PB} J.S. Provan, M.O. Ball, The complexity of counting cuts and of computing the probability that a graph is connected. \textsl{SIAM J. Comput.} \textbf{12} (1983) 777--788.
 %
\bibitem{rot} U.G. Rothblum, Characterization of stable matchings as extreme points of a polytope. \textsl{Math. Programming} \textbf{54} (1992) 57--67.
  %
\bibitem{RRV} A.E. Roth, U.G. Rothblum, J.H. Vande Vate, Stable matching, optimal assignments and linear programming. \textsl{Math. Oper. Res.} \textbf{18} (1993) 808--828.
 %
  \bibitem{sch} A.~Schrijver, \textsl{Combinatorial Optimization}, Vol.~A, Springer, 2003.
 %
\bibitem{tan} J.~Tan, A necessary and sufficient condition for the existence of a complete stable matching. \textsl{J. Algorithms} \textbf{12} (1991) 154--178.
   %
\bibitem{tar} \`E. Tardos, A strongly polynomial algorithm to solve combinatorial linear problems. \textsl{Oper. Research} \textbf{34} (1986)  250--256.
 %
\bibitem{TS}
C.P. Teo, J. Sethuraman. The geometry of fractional stable matchings and its applications. \textsl{Math. Oper. Res.} \textbf{23} (4) (1998) 874--891.
 %
\bibitem{val} L.G. Valiant, The complexity of computing the permanent. \textsl{Theoret. Comp. Sci.} \textbf{8} (1979) 189--201.
%
 \bibitem{VV} J.H. Vande Vate, Linear programming brings marital bliss. \textsl{Oper. Res. Lett.} \textbf{8} (1989) 147--153. 
 

\end{thebibliography}
\end{document}